\documentclass[12pt]{amsart}

\usepackage{graphicx}

\usepackage{subcaption} %para facer a artistada coas figuras nun grid

\usepackage{tikz-cd}
\usepackage{hyperref}
\hypersetup{bookmarks=true,
	unicode=true,
	colorlinks=true,
	citecolor=black,
	linkcolor=black,
	urlcolor=black,
	plainpages=false,
	pdfpagelabels=true}

\usepackage{url}	%url adresses
\allowdisplaybreaks %align may cross pages

\usepackage{pgf}

\usepackage{xcolor}

\usepackage{comment} % para comentar trozos grandes de codigo

\usepackage{tikz}
\usetikzlibrary{arrows}
\usetikzlibrary{decorations.markings,arrows,automata,arrows,backgrounds,snakes}
\usepackage{tikz-cd} % para os diagraminhas
\usepackage{pgf}
\usetikzlibrary{babel}

\usepackage{appendix}

\usepackage{mathrsfs}
\usepackage{dsfont}

\usepackage{mathtools}

\usepackage{accents}

\usepackage[shortlabels]{enumitem}
\usepackage{txfonts}

\newtheorem{theorem}{Theorem}[section]

\newtheorem{lemma}[theorem]{Lemma}

\theoremstyle{definition}
\newtheorem{definition}[theorem]{Definition}
\newtheorem{example}[theorem]{Example}

\theoremstyle{remark}
\newtheorem{remark}[theorem]{Remark}

\usepackage{color}
\definecolor{darkgreen}{cmyk}{1,0,1,.2}
\definecolor{m}{rgb}{1,0.1,1}
\newdimen\theight
\def\TeXref#1{%
	\leavevmode\vadjust{\setbox0=\hbox{{\tt
				\quad\quad  {\small \textrm #1}}}%
		\theight=\ht0
		\advance\theight by \lineskip
		\kern -\theight \vbox to
		\theight{\rightline{\rlap{\box0}}%
			\vss}%
}}%


\begin{document}

\thanks{The author was partially supported by the grant ED431C 2023/31 (Xunta de Galicia, FEDER) and by Programa de axudas á etapa predoutoral da Xunta de Galicia.
 }
    
	\title{Singular multivalued homology}

	\author[A. Majadas-Moure 
	]{%
		Alejandro O. Majadas-Moure  %\and David Mosquera-Lois
		%\and
		%etc.
	}

	\address{
		Alejandro O. Majadas-Moure \\
		Departamento de Matemáticas, Universidade de Santiago de Compostela, Spain}
	\email{alejandro.majadas@usc.es}

	\begin{abstract} 
		Let $X$ be a compact, Hausdorff topological space. Then $H^M_n(X)=0$ for all $n>0$, where $H^M$ is the multivalued analogue of singular homology. The case $n=1$ is already known \cite{St}.
	\end{abstract}
	
	%\subjclass[2020]{
	
	%}
	
	\maketitle
	\section{Introduction}
    Lefschetz fixed-point theorem is one of the main results in fixed-point theory. The study of fixed points of continuous multivalued maps between topological spaces has been an important subject in mathematics since long ago \cite{EM, Ka, ONeil}. Since homology plays an important role in classical fixed-point theory of continuous maps (for example through Lefschetz or Brouwer fixed-point theorems), we can ask whether a ``multivalued singular homology'' (that is, the singular $n$-simplices would be continuous multivalued maps $\sigma:\Delta_n\rightarrow X$, where $\Delta_n\subset\mathbb{R}^{n+1}$ is the standard $n$-simplex) may be useful to study fixed-points of multivalued maps.

    In algebraic geometry, it is usual to find the Lefschetz trace theorem in the context of correspondences (multivalued maps) \cite{Grothendieck, Va}. Also in algebraic geometry, there exists a very important singular homology where the singular simplices are correspondences between the (algebraic) standard simplex and $X$, called Suslin homology  \cite{SV}.

    Hence, back in algebraic topology, we can ask if it is possible to construct a singular homology of topological spaces with multivalued simplices and, in an affirmative case, what its meaning is. More precisely, let $f:X\rightarrow Y$ be a map between two topological spaces. Let $\mathrm{gr}(f)=\{(x,y)\in X\times Y| f(x)=y\}$ be its graph and $\tau_{|\mathrm{gr}(f)}:\mathrm{gr}(f)\rightarrow X$ the restriction of the canonic projection $X\times Y\rightarrow X$. The following are equivalent:
    \begin{enumerate}
        \item f is continuous.
        \item the map $\tau_{|\mathrm{gr}(f)}$ is proper \cite[Chapitre I, S. 10]{Bo}.
    \end{enumerate}
    Thus, we can use this characterization to define a continuous multivalued map between topological spaces $X$ and $Y$ as a subset $T\subset X \times Y$, such that the restriction $\tau_{|T}:T\rightarrow X$ of the canonical projection $X\times Y\rightarrow X$ is proper, surjective, and with finite fibers. This is precisely the translation of the correspondences in \cite{SV} to topology. Hence, the problem consists of finding a  singular homology of a topological space $X$ defined with continuous multivalued maps $\sigma:\Delta_n\rightarrow X$.

    We will show that, though it is possible to construct such a homology $H^M_n$, if $X$ is compact and Hausdorff, then $H_n^M(X)=0$ for all $n>0$ (Theorem~\ref{teor principal}). The particular case of $n=1$ was already addressed in \cite[Corollary 1.1]{St}. Note however that the definition of continuous multivalued map used in \cite{St} is different from our definition (for example, if $X$ is not Hausdorff, the definition used in \cite{St} does not extend the usual definition of continuous maps since $\mathrm{id}:X\rightarrow X$ is not continuous) \cite[Theorem 2]{St2}.

    In this paper we also define the multivalued homotopy groups and show that they also vanish for $n>0$. Finally we obtain a similar fixed-point result to \cite[Theorem 4]{St1}.

\section{Continuous multivalued maps}
\begin{definition}
    A continuous map $f:X\rightarrow Y$ between two topological spaces is \textit{proper} if it is universally closed, that is, if for every continuous map $Y\rightarrow Y'$, the induced map $f':X\times_Y Y'\rightarrow Y'$ is closed. 
\end{definition}
\begin{remark}
The previous definition is equivalent to the one in \cite[Chapitre I, S. 10]{Bo}, that says that $f$ is proper if $X\times Z\rightarrow Y\times Z$ is closed for every topological space $Z$. On the one hand, the first definition implies the second one (given $Z$, we consider $Y'=Y\times Z$). On the other hand, if $f:X\rightarrow Y$ verifies the condition of \cite{Bo}, in particular, $f$ must be closed and then in every diagram 
\[\begin{tikzcd}
	{X\times_Y Y'} & {Y'} \\
	X & Y
	\arrow["q", from=1-1, to=1-2]
	\arrow[from=1-1, to=2-1]
	\arrow["g", from=1-2, to=2-2]
	\arrow["f", from=2-1, to=2-2]
\end{tikzcd}\]
we see that $\mathrm{Im}(q)=g^{-1}(\mathrm{Im}(f))$ is closed in $Y'$. Considering now the diagram 
\[\begin{tikzcd}
	{X\times_Y Y'} & {Y\times_Y Y'=Y'} \\
	{X\times Y'} & {Y\times Y'}
	\arrow["q", from=1-1, to=1-2]
	\arrow[hook, from=1-1, to=2-1]
	\arrow["\beta", hook, from=1-2, to=2-2]
	\arrow["{f\times \mathrm{id}}", from=2-1, to=2-2]
\end{tikzcd}\]
we see that $\mathrm{Im(q)\cap \beta^{-1}(f\times \mathrm{id})(F)}$ is an intersection of two closed sets for every $F$ closed in $X\times Y'$. If $F'=(X\times_Y Y')\cap F$ is closed in $X\times_Y Y'$, we have
\begin{equation*}
    q(F')=\mathrm{Im}(q)\cap \beta^{-1}(f\times \mathrm{id})(F),
\end{equation*}
since, if $q(x_1,y_1)\in \beta^{-1}(f\times\mathrm{id})(F)$, there exists $(x_2,y_2)\in F$ such that 
\begin{equation*}
(f(x_2),y_2)=\beta q(x_1,y_1)=(f(x_1), y_1),
\end{equation*}
and then $y_1=y_2$ and $f(x_1)=f(x_2)$. Since $(x_1,y_1)\in X\times_Y Y'$, we have $f(x_1)=g(y_1)$ and then $f(x_2)=g(y_2)$, so $(x_2,y_2)\in X\times_Y Y'$. Hence $(x_2,y_2)\in F'$ and $q(x_1,y_1)=q(x_2,y_2)$. The other inclusion is obvious and this shows that $q(F')$ is closed in $Y'$.
\end{remark}
Due to the previous remark, from now, we will cite some results in \cite[Chapitre I, S. 10]{Bo} concerning proper maps when necessary. In particular, $f$ is proper iff it is closed and its fibers $f^{-1}(y)$ are compact for every $y\in Y$ \cite[Chapitre I, S. 10, n. 2, Théorème 1]{Bo}.
\begin{definition}
    A \textit{continuous multivalued map} $X\rightarrow Y$ is a subset $T$ of $X\times Y$ such that the restriction $\tau_{|T}:T\rightarrow X$ of the canonical projection $X\times Y \rightarrow X$ is proper, surjective and with finite fibers (equivalently $\tau_{|T}$ closed, surjective and with finite fibers, since every finite fiber is compact).
\end{definition}
If $f:X\rightarrow Y$ is a single-valued map, its graph $T:=\mathrm{gr}(f)=\{(x,f(x))\in X\times Y\}$ is a multivalued map and it is continuous iff $f$ is continuous as a single-valued map (since $\tau_{|\mathrm{gr}(f)}:\mathrm{gr}(f)\rightarrow X$ is a homeomorphism). We will denote the set of continuous multivalued maps from $X$ to $Y$ by $M(X,Y)$.
We will use the following properties of continuous multivalued maps.
\begin{lemma}\label{lema composicion}
    Let $X,Y, Z$ topological spaces and $R\in M(X,Y)$, $S\in M(Y,Z)$. Let $S\circ R$ be the composition of $R$ with $S$, \textit{i.e.}, in terms of the graphs, $S\circ R$ is the image of $(R\times Z)\cap (X\times S)$ by the canonical projection $X\times Y\times Z\rightarrow X\times Z$. Then $S\circ R\in M(X,Z)$. This composition is clearly associative.
\end{lemma}
\begin{proof}
    We must verify that $\tau_{|S\circ R}$ is proper (all the other things are straightforward). Since $\tau_{|S}$ is proper, $\mathrm{id}_X\times \tau_{|S}:X\times S\rightarrow X\times Y$ is also proper and then 
    \begin{equation*}
        (\mathrm{id}_X\times \tau_{|S})_{|(\mathrm{id}_X\times \tau_{|S})^{-1}(R)}:(R\times Z)\cap (X\times S)=(\mathrm{id}_X\times \tau_{|S})^{-1}(R)\rightarrow R
    \end{equation*}
    will also be proper, as follows from the diagram 
    \[\begin{tikzcd}
	{(\mathrm{id}_X\times \tau_{|S})^{-1}(R)=(X\times S)\times_{X\times Y} R} & R \\
	{X\times S} & {X\times Y}.
	\arrow[from=1-1, to=1-2]
	\arrow[from=1-1, to=2-1]
	\arrow[hook, from=1-2, to=2-2]
	\arrow["{\mathrm{id}_X\times \tau_{|S}}", from=2-1, to=2-2]
\end{tikzcd}\]
Now, since $\tau_{|R}:R\rightarrow X$ is proper and the composition of proper maps is also proper  \cite[Chapitre I, S. 10, n. 1, Prop. 5]{Bo}, the composition of these two maps
\begin{equation*}
    (R\times Z)\cap (X\times S)\rightarrow R\rightarrow X
\end{equation*}
is proper. But this composition equals the composition
\begin{equation*}
    (R\times Z)\cap(X\times S)\overset{\alpha}{\rightarrow} S\circ R\overset{\tau_{|S\circ R}}{\rightarrow} X,
\end{equation*}
where $\alpha$ is the restriction of the canonical projection $X\times Y\times Z\rightarrow X\times Z$. Since $\alpha$ is surjective, $\tau_{|S\circ R}$ is proper \cite[Chapitre I, S. 10, n.1, Prop. 5]{Bo}.
\end{proof}
\begin{example}\label{ejemplo}
    Let $X,Y,Z$ be topological spaces, $f:X\rightarrow Y$ a continuous map and $S\in M(Y,Z)$. Let us consider
    \[\begin{tikzcd}
	{(f\times \mathrm{id}_Z)^{-1}(S)=(X\times Z)\times_{Y\times Z}S} & S \\
	{X\times Z} & {Y\times Z.}
	\arrow[from=1-1, to=1-2]
	\arrow[hook, from=1-1, to=2-1]
	\arrow[hook, from=1-2, to=2-2]
	\arrow["{f\times\mathrm{id}_Z}", from=2-1, to=2-2]
\end{tikzcd}\]
We have $(f\times \mathrm{id}_Z)^{-1}(S)=S\circ\mathrm{gr}(f)$, since both sets agree with 
\begin{equation*}
    \{(x,z)\in X\times Z|(f(x),z)\in S\}. 
\end{equation*}
This construction establishes a contravariant functor $M(-,Z)$ from the category of topological spaces to the one of sets.

In a similar way, if $R\in M(X,Y)$ and $g:Y\rightarrow Z$ is a continuous map, we have that $\mathrm{gr}(g)\circ R$ is the image of $R$ by $\mathrm{id}_X\times g:X\times Y\rightarrow X\times Z$.
\end{example}

\begin{definition}
    Let $X,Y,X',Y'$ be topological spaces, $R\in M(X,Y)$ and $S\in M(X',Y')$. We define $R\boxtimes R'$ as
    \begin{equation*}
        R\boxtimes R'=\{(x,x',y,y')\in X\times X'\times Y\times Y'| (x,y)\in R, (x',y')\in R'\}.
    \end{equation*}
    We have that $R\boxtimes R'\in M(X\times X', Y\times Y')$ since $\tau_{|R\boxtimes R'}$ is the composition of the proper maps \cite[Chapitre I, S. 10, n. 1, Prop. 4]{Bo}
    \begin{equation*}
        R\boxtimes R'\rightarrow R\times R'\overset{\tau_{|R}\times\tau_{|R'}}{\rightarrow}X\times X',
    \end{equation*}
    where the first map is the obvious homeomorphism.
\end{definition}
\begin{lemma}\label{lema operaciones}
    Let $X,Y,Z,X',Y',Z'$ be topological spaces, $R\in M(X,Y)$, $S\in M(Y, Z)$, $R'\in M(X',Y')$ and $S'\in M(Y',Z')$. Then,
    \begin{equation*}
        (S\circ R)\boxtimes(S'\circ R')=(S\boxtimes S')\circ (R\boxtimes R').
    \end{equation*}
\end{lemma}
\section{Multivalued singular homology}

\begin{definition}\label{def caras y homo}
    Let $\Delta_n=\{(t_0,\ldots, t_n)\subset \mathbb{R}^{n+1}|0\leq t_i\leq 1, \sum t_i=1 \}$ be the standard $n$-simplex. For each $i\in\{0,\ldots, n\}$, let $\delta_i^n:\Delta_{n-1}\rightarrow \Delta_n$ be the face maps, defined by $\delta_i^n(t_0,\ldots,t_{n-1})=(t_0,\ldots,t_{i-1},0,t_i,\ldots,t_{i-1})$. Given a topological space $X$, we define $S_n^M(X)=M(\Delta_n,X)$, 
    \begin{equation*}
        d_n^i=M(\delta_i^n, X):S_n^M(X)\rightarrow S^M_{n-1}(X)        
    \end{equation*}
    (see Example~\ref{ejemplo}), and $C_n^M(X)$ as the free abelian group with basis $S_n^M(X)$. We obtain a complex $(C_\ast^M(X),d_\ast)$ where $d_n$ is induced by $\sum_{i=0}^n(-1)^id_n^i$.

    For each $n\geq 0$ we also define
    \begin{equation*}
        H_n^M(X)=H_n(C_\ast^M,d_\ast).
    \end{equation*}
\end{definition}
If $R\in M(X,Y)$, by Lemma~\ref{lema composicion} there is a homomorphism of abelian groups $H_n(R):H_n^M(X)\rightarrow H_n^M(Y)$ induced by $S_\ast^M(R):S_\ast^M(X)\rightarrow S_\ast^M(Y)$.
\begin{definition}\label{def homotopia}
    Let $R,S\in M(X,Y)$. We say that $R$ and $S$ are $M$-homotopic if it exists $L\in M(X\times I, Y)$ ($I$ denotes the unit interval) such that $L_{|X\times\{0\}}=R$, $L_{|X\times\{1\}}=S$, where $L_{|X\times\{t\}}$ is the composition of $L$ with the inclusion $X\simeq X\times\{t\}\hookrightarrow X\times I$.
\end{definition}
\begin{lemma}\label{lema principal}
    If $R,S\in M(X,Y)$ are $M$-homotopic, then $H_n^M(R)=H_n^M(S)$ for each $n\geq 0$.
\end{lemma}
\begin{proof}
    The proof is a simple adaptation of the usual proofs to the multivalued context. We will follow, for example, \cite[Section 5.3.1]{Zi}, since these details will be useful for Section~\ref{sec grupos}.

    Let $L\in M(X\times I,Y)$ as in Definition~\ref{def homotopia}. For each $i\in\{0,\ldots,n\}$, let $r_n^i:\Delta_{n+1}\rightarrow \Delta_n\times I$ be defined by
    \begin{equation*}
        r_n^i(t_0,\ldots,t_{n+1})=((t_0,\ldots,t_{i-1},t_i+t_{i+1},t_{i+2},\ldots,t_{n+1}),t_{i+1}+t_{i+2}+\ldots+t_{n+1}).
    \end{equation*}
    The following equalities are easy to check \cite[5.1.3.4, 5.1.3.5]{Zi}:
\begin{align*}
    &(1)\;r_n^{j+1}\circ\delta_{n+1}^i=(\delta_n^i\times \mathrm{id}_I)\circ r^j_{n-1},\; \text{if}\; 0\leq i \leq j \leq n-1,\\
    &(2)\;r_n^{i+1}\circ  \delta_{n+1}^{i+1}=r_n^i\circ\delta_{n+1}^{i+1},\;\text{if}\; i<n,\\
    &(3)\;r_n^i\circ \delta_{n+1}^{j+1}=(\delta_n^j\times\mathrm{id}_I)\circ r^i_{n-1},\;\text{if}\; i<j,\\
    &(4)\;r_n^0\circ\delta_{n+1}^0(e)=(e,1),\;e\in\Delta_n,\\
    &(5)\;r_n^n\circ\delta_{n+1}^{n+1}(e)=(e,0),\;e\in\Delta_n.
\end{align*}
   Now we define for each $n$, $h_n^i:S_n^M(X)\rightarrow S_{n+1}^M(Y)$ by 
   \begin{equation*}
       h_n^i(\alpha)=L\circ(\alpha\boxtimes \mathrm{id}_I)\circ r_n^i.
   \end{equation*}
   From the previous equalities and Lemmas~\ref{lema composicion} and \ref{lema operaciones} we obtain:
   \begin{align*}
       &(1)\;d_{n+1}^i\circ h_n^{j+1}=h_{n-1}^j\circ d_n^i,\;\text{if}\; i\leq j,\\
       &(2)\;d_{n+1}^{i+1}\circ h_n^{i+1}=d_{n+1}^{i+1}\circ h_n^i, \;\text{if}\; i<n,\\
       &(3)\;d_{n+1}^{j+1}\circ h_n^i=h_{n-1}^i\circ d_n^j,\;\text{if}\; i<j,\\
       &(4)\; d_{n+1}^0\circ h_n^0=S_n^M(S),\\
       &(5)\; d_{n+1}^{n+1}\circ h_n^n=S_n^M(R).
   \end{align*}
Let $h_n:C^M_n(X)\rightarrow C_{n+1}^M(Y)$ be the group homomorphism that extends the map from $S_n^M(X)$ to $S_{n+1}^M(Y)$ that sends $\alpha$ to $\sum_{i=0}^n(-1)^ih_n^i(\alpha)$.
   It is easy to check that 
   \begin{equation*}
       d_{n+1}\circ h_n + h_{n-1}\circ d_n=C_n^M(S)-C_n^M(R)
   \end{equation*}
   and then $h_\ast$ is a homotopy of complexes of abelian groups.
\end{proof}
\begin{definition}
    Let $X, Y$ be topological spaces and $A\subset Y$ a finite subset. We define $C_A^X:X\rightarrow Y$ as the multivalued constant map with value $A$. It is continuous \cite[Chapitre I, S. 10, n. 2 Cor. 5 du Théorème 1]{Bo}.
\end{definition}
\begin{lemma}\label{lema anulacion}
    With the notation of the previous definition, the homomorphism
    \begin{equation*}
        H_n^M(C_A^X):H_n^M(X)\rightarrow H_n^M(Y)
    \end{equation*}
    is zero for every $n>0$.
\end{lemma}
\begin{proof}
    Let $\sum_im_i\alpha_i\in C_n^M(X)$ ($m_i\in \mathbb{Z}$, $\alpha_i\in S_n^M(X)$) be a cycle. Its image by $C_A^X$ in $Y$ is the cycle $\sum_i m_i C_A^{\Delta_n}=m C_A^{\Delta_n}$, where $m=\sum_im_i$. We obtain then
    \begin{equation*}
        0=d(mC_A^{\Delta_n})=\sum_{j=0}^n(-1)^jd_n^j(mC_A^{\Delta_n})=m\sum_{j=0}^n(-1)^jC_A^{\Delta_{n-1}},
    \end{equation*}
    so $n$ is odd. Thus, this cycle is a boundary:
    \begin{equation*}
    d(mC_A^{\Delta_{n+1}})=m\sum_{j=0}^{n+1}(-1)^jC_A^{\Delta_n}=mC_A^{\Delta_n}.
    \end{equation*}
\end{proof}
Finally, we obtain that all these positive homology groups vanish.
\begin{theorem}\label{teor principal}
    Let $X$ be a compact and Hausdorff topological space. Then $H_n^M(X)=0$ for all $n>0$.
\end{theorem}
\begin{proof}
    Let $x_0\in X$ and $C=C_{\{x_0\}}^X:X\rightarrow X$ the constant map of value $x_0$. By Lemmas~\ref{lema principal} and \ref{lema anulacion}, it suffices to build an $M$-homotopy between $C$ and $\mathrm{id}_X$. Let us consider the map $L':X\times I\rightarrow X$ defined by:
        \[ L'(x,t) = \begin{cases} 
          x & t< 1/2 \\
          x_0 & t\geq 1/2 
       \end{cases}
    \]
Let $L\in M(X\times I, X)$ be the closure of $\mathrm{gr}(L')$ in $X\times I\times X$. The map $\tau_{|L}:L\rightarrow X\times I$ is proper, since $L$ is compact and $X\times I$ is Hausdorff \cite[Chapitre I, S. 10, n.2, Cor. 2 du Théorème 1]{Bo} and then $L\in M(X\times I, X)$ ($\tau_{|L}$ is also surjective and the fibers are finite). Moreover, 
\begin{equation*}
    L_{|X\times\{0\}}=\mathrm{id}_X,\; \text{and}\; L_{|X\times\{1\}}=C.
\end{equation*}
\end{proof}
\begin{remark}
    The proof of Theorem~\ref{teor principal} shows a method to build $M$-homotopies between two continuous multivalued maps. However, in the particular case of this proof, the $M$-homotopy $L$ is easy to describe directly:
    \begin{equation*}
        L=\mathrm{gr}(L')\cup \{(x,1/2,x)|x\in X\}.
    \end{equation*}
\end{remark}
\section{Multivalued homotopy groups}\label{sec grupos}
In this section, we only note that we can also define homotopy groups with multivalued maps but, as expected, they vanish for compact Hausdorff spaces. First, we notice that $S_\ast^M(X)$ is a simplicial set (the faces have already been defined in Definition~\ref{def caras y homo} and the degeneracy maps are defined similarly). Indeed, it is a Kan complex \cite[1.6]{Ma}. The proof is the usual one \cite[1.5]{Ma} once we consider the following lemma.
\begin{lemma}\label{lema kan}
    Let $X$ and $Y$ be topological spaces and $\{X_i\}_{i=1}^n$ a cover of closed spaces $X_i\subset X$. For each $i\in\{1,\ldots,n\}$, let $T_i\in M(X_i,Y)$ be such that $T_{i|X_i\cap X_j}=T_{j|X_i\cap X_j}$ for all $i,j$. Then, there exists a unique $T\in M(X,Y)$ such that $T_{|X_i}=T_i$ for all $i$.
\end{lemma}
\begin{proof}
We start by defining $T=\cup_{i=1}^nT_i\subset\cup_{i=1}^n(X_i\times Y)=X\times Y$. In order to check that $\tau_{|T}:T\rightarrow X$ is proper, for each continuous map $Z\rightarrow X$ we have a commutative diagram
\[\begin{tikzcd}
	{T_i\times_X Z} & {T\times_X Z} & Z \\
	{T_i} & T & X.
	\arrow[hook, from=1-1, to=1-2]
	\arrow[from=1-1, to=2-1]
	\arrow[from=1-2, to=1-3]
	\arrow[from=1-2, to=2-2]
	\arrow[from=1-3, to=2-3]
	\arrow[hook, from=2-1, to=2-2]
	\arrow[from=2-2, to=2-3]
\end{tikzcd}\]
Sine $T_i\rightarrow X_i$ and $X_i\hookrightarrow X$ are proper (the second one because it is the inclusion of a closed set), we have that $T_i\rightarrow X$ is proper and hence the composition of the upper line in the diagram is a closed map. Since $\{T_i\times_XZ\}_{i}$ is a finite cover of $T\times_X Z$, the map $T\times_X Z\rightarrow Z$ is closed.
\end{proof}
\begin{definition}
We consider now a $0$-simplex in $S_\ast^M(X)$. This is an element of $S_0^M(X)$ that we can identify with a finite subset $A\subset X$. We define the \textit{multivalued homotopy groups} of $X$ as $\pi_n^M(X,A)=\pi_n(S^M_\ast(X),A)$ for each $n>0$ \cite[3.6]{Ma}.
\end{definition}
As in the usual case, if $X$ is path-connected and $A,B\subset X$ are finite subsets of $X$,
\begin{equation*}
    \pi_n^M(X,A)\simeq\pi_n^M(X,B).
\end{equation*}
This can be deduced from \cite[Prop. I.8.1]{GJ} using the following lemma.
\begin{lemma}
    Let $A,B$ be finite subsets of a path-connected space $X$. Then, there is a multivalued path joining them, that is, there exists $T\in M(I,X)=S_1^M(X)$ such that its faces are $A$ and $B$.
\end{lemma}
\begin{proof}
    For each $a\in A$ and $b\in B$, we take $f_{ab}:I\rightarrow X$ a map such that $f_{ab}(0)=a$ and $f_{ab}(1)=b$. Now we define
    \begin{equation*}
        T=\cup_{a\in A,\,b\in B}\mathrm{gr}(f_{ab})\subset I\times X.
    \end{equation*}
    It is easy to check that, with this definition, $T\in M(I,X)$ (the proof is similar to the one of Lemma~\ref{lema kan}).
\end{proof}
Again, if we write the proof of Theorem~\ref{teor principal} in terms of  homotopy of simplicial sets (the proof of Lemma~\ref{lema principal} shows that an $M$-homotopy induces a homotopy of maps of simplicial sets), we obtain $\pi_n^M(X,A)=0$ for each $n>0$ if $X$ is compact and Hausdorff.
\section{A final remark in fixed-point theory}
Even if multivalued homology does not allow to detect fixed points with a kind of Lefschetz theorem, we can adapt the proof of \cite[Theorem 4]{St1} to detect fixed sets.
\begin{definition}
    Let $T:X\rightarrow X$ be a continuous multivalued map. We say that a subset $A\subset X$ is a \textit{fixed subset} if $A=T(A)$ ($T(A)$ is the projection on the second factor $X$ of $T\cap(A\times X)$).
\end{definition}
We start with an analogue to \cite[Lemma 4]{St1}:
\begin{lemma}\label{lema analogo 4}
    Let $T:X\rightarrow Y$ be a continuous multivalued map, with $X$ a compact space and $Y$ a compact and Hausdorff space. Then, $T(A)$ is closed for every closed subspace $A\subset X$.
\end{lemma}
\begin{proof}
    Since $\tau_{|T}:T\rightarrow X$ is proper and $X$ is compact, then $T=(\tau_{|T})^{-1}(X)$ is compact \cite[Chapitre I, S. 10, n. 3, Prop. 6]{Bo}. Now, if $A\subset X$ is a closed subspace, it is compact and hence $T\cap(A\times Y)$ is also compact. Then, the projection of this term on $Y$ under the second projection is also compact and so it must be closed since $Y$ is Hausdorff.
\end{proof}
We also have an analogue to \cite[Lemma 3]{St1}:
\begin{lemma}\label{lema analogo 3}
    Let $T:X\rightarrow Y$ be a continuous multivalued map with $X$ a compact Hausdorff space and $Y$ a Hausdorff space. Let $\{x_n\}_{n\in\mathbb{N}^+}$ be a sequence in $X$ that converges to a point $x_0$ and suppose that there is a sequence $y_n\in T(x_n)$ that converges to a point $y_0\in Y$. Then $y_0\in T(x_0)$.
\end{lemma}
\begin{proof}
    Since $T$ is closed in $X\times Y$ (it is a compact space of a Hausdorff space) and $\{(x_n,y_n)\}_{n\in\mathbb{N}^+}\subset T$, we conclude that $\mathrm{lim}(x_n,y_n)=(x_0,y_0)\in T$.
\end{proof}
Now the proof of the next theorem is analogous to the one in \cite[Theorem 4]{St1}.
\begin{theorem}
    Every compact Hausdorff space has the fixed set property for a continuous multivalued map, \textit{i.e.}, there is always a fixed subset.
\end{theorem}
	
\end{document}